\documentclass[11pt]{article}

\usepackage{amsthm,amsmath,amssymb}

\usepackage{graphicx}

\usepackage[colorlinks=true,citecolor=black,linkcolor=black,urlcolor=blue]{hyperref}

\theoremstyle{plain}
\newtheorem{theorem}{Theorem}
\newtheorem{lemma}[theorem]{Lemma}
\newtheorem{corollary}[theorem]{Corollary}

\theoremstyle{definition}

\theoremstyle{remark}

\title{\bf Complementary Families of the Fibonacci-Lucas Relations}

\author{Ivica Martinjak\\
Faculty of Science, University of Zagreb\\
Bijeni\v cka cesta 32, HR-10000 Zagreb, Croatia\\
\and
Helmut Prodinger\\
Department of Mathematics, University of Stellenbosch\\
7602 Stellenbosch, South Africa\\
}

\begin{document}

\maketitle

\begin{abstract}
In this paper we present two families of Fibonacci-Lucas identities, with the Sury's identity being the best known representative of one of the families. While these results can be proved by means of the basic identity relating Fibonacci and Lucas sequences we also provide a bijective proof. Both families are then treated by generating functions.
\end{abstract}


\section{Introduction} 

\noindent
The Fibonacci and Lucas sequences of numbers are defined as 
\begin{align}  \label{FiboRecurr}
F_{n+2}=F_{n+1}+F_{n}, \enspace F_0=0, \enspace F_1=1,\\
L_{n+2}=L_{n+1}+L_{n}, \enspace L_0=2, \enspace L_1=1,
\end{align}
where $n \in \mathbb{N}_0$, and denoted by $(F_n)_{n \ge 0}$ and $(L_n)_{n \ge 0}$, respectively. Equivalently, these sequences could be defined as the only solutions $(x,y)$, $x=L_n$, $y=F_n$ of the Diophantine equation $$x^2-5y^2=4(-1)^n.$$ 

For natural numbers $k,n$ we consider the number of ways to tile a $k \times n$ rectangle with certain tiles.  The area of  a  $1 \times 1$ rectangle is usually called a {\it cell}. Tiles that are $1 \times 1$ rectangles are called {\it squares} and tiles of dimension $1 \times 2$ are called {\it dominoes}. In particular, when $k=1$ a rectangle to tile is called $n$-{\it board}. 

Among many interpretations of the Fibonacci and Lucas numbers, here we use the fact that the number $f_n$ of an $n$-board tilings with squares and dominoes is equal to $F_{n+1}$ while the number $l_n$ of the same type of tilings of a circular $n$-board is equal to $L_n$. Namely, an $n$-board tiling begins either with a square or with a domino, meaning that the numbers $f_n$ obey the same recurrence relation as the Fibonacci numbers. Having in mind that the sequence $(f_n)_{n \ge 0}$ begins with $f_1=1$, $f_2=2$ we get 
\begin{align} \label{Fnfn}
f_n=F_{n+1}. 
\end{align}
Similar reasoning proves that the number of tilings of a circular board of the length $n$ is equal to the $n$-th Lucas number,
\begin{align} \label{Lnln}
l_n=L_n.
\end{align}

A circular board of the length $n$ we shall call {\it n-bracelet}. We say that a bracelet is {\it out of phase} when a single domino covers cells $n$ and 1, and {\it in phase} otherwise. Consequently there are more ways to tile an $n$-bracelet then an $n$-board. More precisely, $n$-bracelet tilings are equinumerous to the sum of tilings of an $n$-board and $(n-2)$-board, 
\begin{align} \label{FiboLucasElemId1}
l_n=f_n+f_{n-2}. 
\end{align}
This follows from the fact that in phase tilings can be unfolded into $n$-board tilings while out of phase tilings can be straightened into $(n-2)$-board tilings since a single domino is fixed on cells $n$ and 1. Once having (\ref{Fnfn}) and (\ref{Lnln}), the statement of Lemma \ref{FiboLucasElemId} follows immediately from relation (\ref{FiboLucasElemId1}).

\begin{lemma}\label{FiboLucasElemId}
The $n$-th Lucas number is equal to the sum of $(n-1)$-th and $(n+1)$-th Fibonacci number,
\begin{align}
L_n=F_{n-1} + F_{n+1}.
\end{align}
\end{lemma}

Lemma \ref{FiboLucasElemId} presents the most elementary identity involving both the Fibonacci and Lucas sequence. In what follows we use it in order to prove further identities. It is worth mentioning that there are numerous identities known for these sequences. Many of them one can find in the classic reference \cite{vajda}. An introduction to Fibonacci polynomials can be seen in \cite{GKP} while recent results on this subject one can find in \cite{ACMS}.

\section{Colored tilings and the product $m^nF_{n+1}$} 

A tiling either of a board or bracelet of length $n$ with squares in $m$ colors and dominoes in $m^2$ colors we shall call $(n,m)$-tiling. We let $c_1,c_2,\ldots, c_m$ denote colors of squares and we let $c'_1,c'_2,\ldots, c'_{m^2} $ denote colors of dominoes within an $(n,m)$-tiling. In particular, when $m=2$ we choose white and black as colors $c_1$ and $c_2$, respectively. In case $m=3$ we choose white, gray and black for $c_1, c_2, c_3$, respectively. 

One can easily check that there are 8 tilings of a 2-board with squares in two colors and dominoes in four colors (Figure \ref{Fig1}). Furthermore, there are 24 such tilings of a 3-board, 80 tilings of a 4-boards, \ldots In general, there are $2^nf_n$ such tilings of an $n$-board. The numbers $1, 3,18,81,\ldots,3^nf_n,\ldots$ also have a combinatorial interpretation. They count the ways to tile an $n$-board with squares in 3 colors and dominoes in 9 colors. We generalize these facts in the following

\begin{lemma}
The product $m^n f_n$ represents the number of colored $n$-board tilings with squares in $m$ colors and dominoes in $m^2$ colors.   
\end{lemma}

\begin{proof} There are $f_n$ uncolored tilings of an $n$-board with squares and dominoes. Tilings of two neighbouring cells with squares in $m$ colors are equinumerous to those with dominoes in $m^2$ colors. Thus, each of $f_n$ tilings gives $m^n$ tilings when we use $m$ colors for squares and $m^2$ colors for dominoes.
\end{proof}

The same argument proves that $m^nL_n$ represents the number of colored $n$-bracelet tilings with squares in $m$ colors and dominoes in $m^2$ colors. 

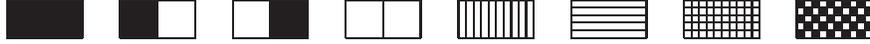
\begin{figure}[h!]
				
\setlength{\unitlength}{0.5cm}

\begin{picture}(0,1)(-2.5,0)
\put(0,0){\line(1,0){2}}
\put(2,0){\line(0,1){1}}
\put(0,0){\line(0,1){1}}
\put(0,1){\line(1,0){2}}
\multiput(0,0)(0.01,0){199} {\line(0,1){1}}

\put(3,0){\line(1,0){2}}
\put(5,0){\line(0,1){1}}
\put(3,0){\line(0,1){1}}
\put(3,1){\line(1,0){2}}
\multiput(3,0)(0.01,0){100} {\line(0,1){1}}

\put(6,0){\line(1,0){2}}
\put(8,0){\line(0,1){1}}
\put(6,0){\line(0,1){1}}
\put(6,1){\line(1,0){2}}
\multiput(7,0)(0.01,0){99} {\line(0,1){1}}

\put(9,0){\line(1,0){2}}
\put(11,0){\line(0,1){1}}
\put(9,0){\line(0,1){1}}
\put(9,1){\line(1,0){2}}
\put(10,0){\line(0,1){1}}

\put(12,0){\line(1,0){2}}
\put(14,0){\line(0,1){1}}
\put(12,0){\line(0,1){1}}
\put(12,1){\line(1,0){2}}
\multiput(12,0)(0.2,0){10}{\line(0,1){1}}

\put(15,0){\line(1,0){2}}
\put(17,0){\line(0,1){1}}
\put(15,0){\line(0,1){1}}
\put(15,1){\line(1,0){2}}
\multiput(15,0)(0,0.2){5}{\line(1,0){2}}

\put(18,0){\line(1,0){2}}
\put(20,0){\line(0,1){1}}
\put(18,0){\line(0,1){1}}
\put(18,1){\line(1,0){2}}
\multiput(18,0)(0.2,0){10}{\line(0,1){1}}
\multiput(18,0)(0,0.2){5}{\line(1,0){2}}

\put(21,0){\line(1,0){2}}
\put(23,0){\line(0,1){1}}
\put(21,0){\line(0,1){1}}
\put(21,1){\line(1,0){2}}
\multiput(21,0)(0.2,0){10}{\line(0,1){1}}
\multiput(21,0)(0,0.2){5}{\line(1,0){2}}

\multiput(21,0)(0.01,0){20}{\line(0,1){0.2}}
\multiput(21.4,0)(0.01,0){20}{\line(0,1){0.2}}
\multiput(21.8,0)(0.01,0){20}{\line(0,1){0.2}}
\multiput(22.2,0)(0.01,0){20}{\line(0,1){0.2}}
\multiput(22.6,0)(0.01,0){20}{\line(0,1){0.2}}

\multiput(21.2,0.2)(0.01,0){20}{\line(0,1){0.2}}
\multiput(21.6,0.2)(0.01,0){20}{\line(0,1){0.2}}
\multiput(22,0.2)(0.01,0){20}{\line(0,1){0.2}}
\multiput(22.4,0.2)(0.01,0){20}{\line(0,1){0.2}}
\multiput(22.8,0.2)(0.01,0){20}{\line(0,1){0.2}}

\multiput(21,0.4)(0.01,0){20}{\line(0,1){0.2}}
\multiput(21.4,0.4)(0.01,0){20}{\line(0,1){0.2}}
\multiput(21.8,0.4)(0.01,0){20}{\line(0,1){0.2}}
\multiput(22.2,0.4)(0.01,0){20}{\line(0,1){0.2}}
\multiput(22.6,0.4)(0.01,0){20}{\line(0,1){0.2}}

\multiput(21.2,0.6)(0.01,0){20}{\line(0,1){0.2}}
\multiput(21.6,0.6)(0.01,0){20}{\line(0,1){0.2}}
\multiput(22,0.6)(0.01,0){20}{\line(0,1){0.2}}
\multiput(22.4,0.6)(0.01,0){20}{\line(0,1){0.2}}
\multiput(22.8,0.6)(0.01,0){20}{\line(0,1){0.2}}

\multiput(21,0.8)(0.01,0){20}{\line(0,1){0.2}}
\multiput(21.4,0.8)(0.01,0){20}{\line(0,1){0.2}}
\multiput(21.8,0.8)(0.01,0){20}{\line(0,1){0.2}}
\multiput(22.2,0.8)(0.01,0){20}{\line(0,1){0.2}}
\multiput(22.6,0.8)(0.01,0){20}{\line(0,1){0.2}}

\end{picture}

\caption{The eight $(2,2)$-tilings of a board of length 2. In general there are $m^nF_{n+1}$ tilings with squares in $m$ colors and dominoes in $m^2$ colors of a board of length $n$. \label{Fig1}}
\end{figure}

There are a few ways to prove Fibonacci-Lucas identity $(\ref{suryId})$. It can be proved using Binet's formula \cite{sury}, and by means of generating function as well \cite{kwong}. Here we continue to call it Sury's identity, as it is already done in \cite{kwong, mart}.

\begin{lemma} \label{suryThm}
For the Fibonacci sequence $(F_n)_{n \ge 0}$ and the Lucas sequence $(L_n)_{n \ge 0}$
\begin{align} \label{suryId}
\sum_{k=0}^n 2^k L_k  = 2^{n+1 } F_{n+1}.
\end{align}
\end{lemma}

\begin{proof} We use Lemma \ref{FiboLucasElemId} and recurrence relation (\ref{FiboRecurr}) to get
\begin{align*}
 L_0&+2L_1+4L_2+\cdots + 2^n L_n\\
&=2F_1+2(F_0+F_2) +4 (F_1+F_3)+ \cdots + 2^n (F_{n-1} +  F_{n+1})\\
&= (2F_1+2F_2+4F_3+ \cdots + 2^nF_{n+1}) + (2F_0 + 4F_1+ 8F_2 + \cdots + 2^nF_{n-1})\\
&= (2F_1 + 2F_0) + (2F_2 + 4F_3+\cdots +2^nF_{n+1}) + (4F_1+ 8F_2+ \cdots + 2^nF_{n-1})\\
&= (2F_2+2F_2) + (4F_3+ \cdots +2^nF_{n+1}) + (4F_1+8F_2+\cdots+2^nF_{n-1})\\
&= (4F_2+4F_1) + (4F_3+ \cdots + 2^nF_{n+1}) + (8F_2+ \cdots + 2^nF_{n-1})\\
&= \cdots\\
&= (2^{n-1}F_n + 2^{n-1} F_{n-1}) + 2^{n-1} F_{n+1}\\
&=2^{n-1}F_{n+1} + 2^{n-1} F_{n+1}\\
&= 2^n F_{n+1}.
\end{align*}
\end{proof}

We let ${\cal A}_{n,m}$ denote the set of all $(n,m)$-board tilings. Furthermore, we let ${\cal A}_{n,m}^{s}$ denote the set of those tilings having a non-white square on $k$-th cell while the cells $k+1$ through $n$ are covered by white squares, $1 \leq k \leq n$. Similary, we let ${\cal A}_{n,m}^{d}$ denote the set of those tilings having a domino on positions $k-1$ and $k$ while the cells $k+1$ through $n$ are covered by white squares. Thus, 
\begin{align*}
{\cal A}_{n,m} = {\cal A}_{n,m}^{s} \cup {\cal A}_{n,m}^{d} \cup \{t_{n,w} \}
\end{align*}
where $t_{n,w}$ is the unique all-white squares tiling of a board of length $n$. We let ${\cal B}_{n,m}$ denote the set of all $(n,m)$-bracelet tilings.  Now we have
\begin{align*}
 {\cal B}_{n,m} =  \bigcup_{i=1}^m {\cal B}_{n,m}^{c_i} \cup  {\cal B}_{n,m}^{p} \cup  {\cal B}_{n,m}^{o}
\end{align*}
where ${\cal B}_{n,m}^{c_i}$ is the set of $(n,m)$-bracelet tilings ending with a square of color $c_i$, the ${\cal B}_{n,m}^{p}$ contains in phase tilings ending with a domino and ${\cal B}_{n,m}^{o}$ contains out of phase tilings. It is worth to note that equalities
\begin{align*}
&|{\cal B}_{n,m}^{c_1}| =|{\cal B}_{n,m}^{c_1}|  = \cdots=  |{\cal B}_{n,m}^{c_m}| \\
&|{\cal B}_{n,m}^{p}| = |{\cal B}_{n,m}^{o}|
\end{align*}
obviously holds true.

There is also an elegant bijective proof of Lemma \ref{suryThm} \cite{benjamin}. It is based on a 1 to 2 correspondence between the set ${\cal B}_{k,2}$ of bracelet tilings and the set ${\cal A}_{k,2}$ of board tilings, $1 \leq k \leq n$. Thus, in both sets squares are colored in two colors (white and black) and dominoes are colored in four colors. 

More precisely, we establish a 1 to 1 correspondence between the sets ${\cal A}_{k,2}^{c_2}$ and ${\cal B}_{k,2}^{c_1}$ as well as between the sets ${\cal A}_{k,2}^{d}$ and ${\cal B}_{n,2}^{p}$, $1 \leq k \leq n$, by the operations {\it i)} removing tiles from $k+1$ through $n$ and {\it ii)} gluing cells $k$ and 1 together.

In the same manner we generate bracelets ending with a black square and out of phase bracelets. Two remaining tilings $t_{n,w}$ consisting from all-white squares are mapped to two $0$-bracelets. Thus, twice the number of $(n,2)$-board tilings is needed to establish correspondence with bracelets of length at most $n$ that are tiled with squares in 2 colors and dominoes with 4 colors. Having in mind $|{\cal A}_{n,2}| = 2^nf_n$, we have 
\begin{align*}
\sum_{k=0}^{n} 2^k L_k = 2 \cdot 2^n f_n= 2^{n+1}F_{n+1},
\end{align*}
which completes the proof.


\section{The main result}

The next Theorem \ref{secondThm} gives an extension of the relation (\ref{suryId}). It provides the answer whether there is an identity involving the product $3^nl_n$, respectively $3^nf_n$ (which appears in some other contexts \cite{ZuDo}). We prove it combinatorially while a proof by means of Lemma \ref{FiboLucasElemId} is also possible.

\begin{theorem}  \label{secondThm}
For the Fibonacci sequence $(F_n)_{n \ge 0}$ and the Lucas sequence $(L_n)_{n \ge 0}$ 
\begin{align} \label{secondThmId}
\sum_{k=0}^n 3^k (L_k + F_{k+1}) = 3^{n+1 } F_{n+1}.
\end{align}
\end{theorem}

\begin{proof}
Note that within $(n,3)$-bracelet tilings the $n$-th cell can be tiled by
\begin{itemize}
\item a square, in one of the colors $c_1, c_2, c_3$,
\item a domino in phase, and a domino out of phase.
\end{itemize}
According to the previouis notations, the sets $ {\cal B}_{n,3}^{c_1} $, $ {\cal B}_{n,3}^{c_2} $  and $ {\cal B}_{n,3}^{c_3} $ denote those tilings ending with a square in a color $c_1,c_2, c_3$, respectively, $ {\cal B}_{n,3}^{p} $ denote tilings ending with a domino in phase and $ {\cal B}_{n,3}^{o} $ denote tilings with a domino out of phase.

The set ${\cal A}_{n,3}$ of $(n,3)$-board tilings we divide into disjoint subsets $ {\cal A}_{n,3}^{s} $, ${\cal A}_{n,3}^{d}$ and $ \{t_{n,w} \}$,
\begin{align*}
{\cal A}_{n,3} = {\cal A}_{n,3}^{s} \cup {\cal A}_{n,3}^{d} \cup \{t_{n,w} \}.
\end{align*}

In the set ${\cal A}_{n,3}^{s}$, there are tilings consisting of a nonwhite square covering cell $k$ and white squares on cells $k+1$ through $n$, $1 \leq k \leq n$. The set ${\cal A}_{n,3}^{d}$ contains tilings having domino on cell $k$ while white squares covers cells $k+1$ through $n$. Now, removing tiles from $k+1$ through $n$ and gluing cells $k$ and 1 together, from ${\cal A}_{n,3}^{s}$ we obtain $(k,3)$-bracelet tilings ending with gray and those tilings ending with a black square. Thus, $$|{\cal A}_{n,3}^{s}| = |{\cal B}_{n,3}^{c_2}| + |{\cal B}_{n,3}^{c_3}|.$$
By the same operations, from the set ${\cal A}_{n,3}^{d}$ we get $(k,3)$-bracelet tilings ending with a domino in phase, $$|{\cal A}_{n,3}^{d}| = |{\cal B}_{n,3}^{p}|.$$ The board tiling $t_{n,w}$ consisting of all-white squares is mapped to the 0-bracelet.

In the same way we establish a correspondence between $(n,3)$-board tilings and out of phase $(k,3)$-bracelet tilings together with twice the number of $(k,3)$-bracelets ending with a white square, 
\begin{align*}
|{\cal A}_{n,3}^{s}| &= |{\cal B}_{n,3}^{c_1}| + |{\cal B}_{n,3}^{c_3}|,\\
|{\cal A}_{n,3}^{d}| &= |{\cal B}_{n,3}^{o}|. 
\end{align*}
The extra set ${\cal B}_{n,3}^{c_3}$ of $(k,3)$-bracelets ending with black square can be unfolded into $(k-1,3)$-board tilings (according to the arguments we use when proving Lemma \ref{FiboLucasElemId}). This means that relation
$$   \sum_{i=1}^{3} |{\cal B}_{n,3}^{c_i}|+|{\cal B}_{n,3}^{p}|+|{\cal B}_{n,3}^{o}| +1 + \sum_{k=0} ^{n-1}3^kF_{k+1}      =   2| {\cal A}_{n,3} |  $$
holds true. Furthermore, we have 
\begin{align*}
\sum_{k=0}^n 3^k L_k + \sum_{k=0} ^{n-1}3^kF_{k+1} &= 2\cdot 3^n F_{n+1}
\end{align*}
and finally 
\begin{align*}
 \sum_{k=0}^n 3^k L_k + \sum_{k=0} ^{n}3^kF_{k+1} &= 3^{n+1} F_{n+1},
\end{align*}
which completes the proof.
\end{proof}

For example, when $n=3$ then 
\begin{align*}
 | {\cal B}_{n,3}^{c_i} | &= 22\ (=18+3+1),\enspace i=1,2,3,\\
|{\cal B}_{n,3}^{p}|=|{\cal B}_{n,3}^{o}|&=36 (=27+9)
\end{align*}
and
\begin{align*}
|{\cal A}_{n,3}^{s}| &= 44 \ (=36+6+2),\\
|{\cal A}_{n,3}^{d}|&=36\  (=27+9).
\end{align*}
Figure \ref{Fig2} illustrates this instance showing $(3,3)$-board tilings in the set ${\cal A}_{3,3}^{s}$ for which $k=2$.

\begin{figure}[h!]
				
\setlength{\unitlength}{0.5cm}

\begin{picture}(0,1)(-2.5,0)
\put(0,0){\line(1,0){3}}
\put(0,0){\line(0,1){1}}
\put(1,0){\line(0,1){1}}
\put(2,0){\line(0,1){1}}
\put(3,0){\line(0,1){1}}
\put(0,1){\line(1,0){3}}
\multiput(1,0)(0.01,0){99} {\line(0,1){1}}

\put(4,0){\line(1,0){3}}
\put(4,0){\line(0,1){1}}
\put(5,0){\line(0,1){1}}
\put(6,0){\line(0,1){1}}
\put(7,0){\line(0,1){1}}
\put(4,1){\line(1,0){3}}
\multiput(4,0)(0.2,0){5}{\line(0,1){1}}
\multiput(4,0)(0,0.2){5}{\line(1,0){1}}

\multiput(5,0)(0.01,0){99} {\line(0,1){1}}

\put(8,0){\line(1,0){3}}
\put(8,0){\line(0,1){1}}
\put(9,0){\line(0,1){1}}
\put(10,0){\line(0,1){1}}
\put(11,0){\line(0,1){1}}
\put(8,1){\line(1,0){3}}
\multiput(9,0)(0.01,0){99} {\line(0,1){1}}

\multiput(8,0)(0.01,0){99} {\line(0,1){1}}

\put(12,0){\line(1,0){3}}
\put(12,0){\line(0,1){1}}
\put(13,0){\line(0,1){1}}
\put(14,0){\line(0,1){1}}
\put(15,0){\line(0,1){1}}
\put(12,1){\line(1,0){3}}
\multiput(13,0)(0.2,0){5}{\line(0,1){1}}
\multiput(13,0)(0,0.2){5}{\line(1,0){1}}

\put(16,0){\line(1,0){3}}
\put(16,0){\line(0,1){1}}
\put(17,0){\line(0,1){1}}
\put(18,0){\line(0,1){1}}
\put(19,0){\line(0,1){1}}
\put(16,1){\line(1,0){3}}
\multiput(17,0)(0.2,0){5}{\line(0,1){1}}
\multiput(17,0)(0,0.2){5}{\line(1,0){1}}

\multiput(16,0)(0.2,0){5}{\line(0,1){1}}
\multiput(16,0)(0,0.2){5}{\line(1,0){1}}

\put(20,0){\line(1,0){3}}
\put(20,0){\line(0,1){1}}
\put(21,0){\line(0,1){1}}
\put(22,0){\line(0,1){1}}
\put(23,0){\line(0,1){1}}
\put(20,1){\line(1,0){3}}
\multiput(21,0)(0.2,0){5}{\line(0,1){1}}
\multiput(21,0)(0,0.2){5}{\line(1,0){1}}

\multiput(20,0)(0.01,0){99} {\line(0,1){1}}

\end{picture}

\caption{The six tilings in the set ${\cal A}_{3,3}^{s}$, for the case $k=2$. \label{Fig2}}
\end{figure}
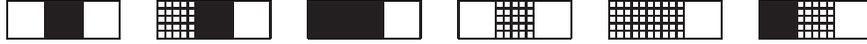

Clearly, there is a 1 to 1 correspondence between $(n,m)$-boards with a domino as the last nonwhite tile and either in phase tilings or out of phase tilings. There is also a 1 to 1 correspondence between boards with a square as the last nonwhite tile and $(k,m)$-bracelets, $1 \leq k \leq n$ ending with square in color $c_i$, $1 \leq i \leq m-1$. In order to establish a 1 to 2 correspondence we have to add $m-2$ sets of $(k-1,m)$-boards to the set of $(k,m)$-bracelets. This reasoning proves Theorem \ref{generalThm}. The identity for the next case $m=4$ is as follows,
\begin{align}
\sum_{k=0}^n 4^k (L_k + 2F_{k+1}) = 4^{n+1 } F_{n+1}.
\end{align}

\begin{theorem} \label{generalThm}
For the Fibonacci sequence $(F_n)_{n \ge 0}$ and the Lucas sequence $(L_n)_{n \ge 0}$ of numbers, with $m \ge 2$
\begin{align} \label{SuryGeneral}
\sum_{k=0}^n m^k [L_k + (m-2)F_{k+1}] = m^{n+1 } F_{n+1}.
\end{align}
\end{theorem}
In Theorem \ref{altSumGeneral} we give a further extension of Theorem \ref{suryThm}, to any $m$. We prove it applying Lemma \ref{FiboLucasElemId}.

\begin{theorem} \label{altSumGeneral}
The alternating sum of products $m^{n-k}$, $0 \leq k \leq n$ with Fibonacci and Lucas numbers is equal to either positive or negative value of $(n+1)$-th Fibonacci number, 
\begin{align} \label{alterIdentGeneral}
\sum_{k=0}^n (-1)^k m^{n-k} [L_{k+1} + (m-2)F_{k}] = (-1)^{n } F_{n+1}.
\end{align}
where $m \ge 2$.
\end{theorem}

\begin{proof}
\begin{align*}
 m^nL_1&-m^{n-1}L_2+m^{n-2}L_3- \cdots +(-1)^nL_{n+1} \\
&= m^n(F_0+F_2)-m^{n-1}(F_1+F_3)+ m^{n-2}(F_2+F_4)- \cdots +(-1)^n(F_n+F_{n+2})\\
&= m^nF_1 - m^{n-1} (mF_1+F_2) + m^{n-2}(mF_2+F_3)- \cdots +(-1)^n(mF_n+F_{n+1})\\
&= m^nF_1-m^nF_1 -m^{n-1}F_2 + m^{n-1}F_2 + m^{n-2}F_3 - m^{n-2}F_3 - \cdots +(-1)^nF_{n+1}\\
&= (-1)^nF_{n+1}.
\end{align*}
\end{proof}

\begin{corollary} \label{FirstThm}
The alternating sum of products $2^{n-k}L_{k+1}$, $0 \leq k \leq n$ is equal to either positive or negative value of $(n+1)$-th Fibonacci number,
\begin{align} \label{alterIdent}
\sum_{k=0}^n (-1)^k 2^{n-k} L_{k+1} = (-1)^{n} F_{n+1}.
\end{align}
\end{corollary}

Thus, Theorem \ref{generalThm} and Theorem \ref{altSumGeneral} gives complementary families of the Fibonacci-Lucas identities, with the Sury identity as the best known representative of (\ref{SuryGeneral}).

\section{A generating functions approach}

Sums as the one in (\ref{secondThmId}) or (\ref{SuryGeneral}) can be systematically evaluated using \emph{generating functions}.
We recall that
\begin{equation*}
\sum_{n\ge0}F_nz^n=\frac{z}{1-z-z^2}\quad\text{and}\quad
\sum_{n\ge0}L_nz^n=\frac{2-z}{1-z-z^2}.
\end{equation*}
Consequently we have
\begin{align*}
\sum_{n\ge0}\bigg(\sum_{0\le k\le n}m^kF_k\bigg)z^n&=\frac{1}{1-z}\frac{mz}{1-mz-m^2z^2}\\
&=\frac{m}{m^2+m-1}\frac{1+m^2z}{1-mz-m^2z^2}-\frac{m}{m^2+m-1}\frac1{1-z}.
\end{align*}
Comparing coefficients, this leads to the explicit formula
\begin{equation*}
\sum_{0\le k\le n}m^kF_k=\frac{m^{n+1}}{m^2+m-1}[F_{n+1}+mF_n]-\frac{m}{m^2+m-1}.
\end{equation*}
Likewise,
\begin{align*}
\sum_{n\ge0}\bigg(\sum_{0\le k\le n}m^kL_k\bigg)z^n&=\frac{1}{1-z}\frac{2-mz}{1-mz-m^2z^2}\\
&=\frac{m}{m^2+m-1}\frac{(2m+1)-m(m-2)z}{1-mz-m^2z^2}+\frac{m-2}{m^2+m-1}\frac1{1-z}
\end{align*}
and
\begin{equation*}
\sum_{0\le k\le n}m^kL_k=\frac{m^{n+1}}{m^2+m-1}[(2m+1)F_{n+1}-(m-2)F_n]+\frac{m-2}{m^2+m-1}.
\end{equation*}

This evaluates the formula  (\ref{SuryGeneral}):
\begin{align*}
\sum_{0\le k\le n}&m^k[L_k+(m-2)F_{k+1}]
=\sum_{0\le k\le n}m^kL_k+\frac{m-2}{m}\sum_{1\le k\le n+1}m^kF_{k}\\
&=\frac{m^{n+1}}{m^2+m-1}[(2m+1)F_{n+1}-(m-2)F_n]+\frac{m-2}{m^2+m-1}\\
&+\frac{m-2}{m}\frac{m^{n+1}}{m^2+m-1}[F_{n+1}+mF_n]-\frac{m-2}{m}\frac{m}{m^2+m-1}+\frac{m-2}{m}m^{n+1}F_{n+1}\\
&=m^{n+1}F_{n+1}.
\end{align*}

Alternating sums, as in (\ref{alterIdentGeneral}), can be handled as well: We consider the generating function
\begin{align*}
\sum_{n\ge0}&\bigg(\sum_{0\le k\le n}(-1)^km^{n-k}F_k\bigg)z^n=
\sum_{k,l\ge0}(-1)^kF_kz^km^{l}z^l\\
&=\frac1{1-mz}\frac{-z}{1+z-z^2}=\frac1{m^2+m-1}\frac{m+z}{1+z-z^2}-\frac m{m^2+m-1}\frac1{1-mz},
\end{align*}
which leads to
\begin{equation*}
\sum_{0\le k\le n}(-1)^km^{n-k}F_k=\frac{(-1)^{n+1}}{m^2+m-1}[mF_{n+1}+F_n]-\frac m{m^2+m-1}m^n.
\end{equation*}
Similarly,
\begin{align*}
\sum_{n\ge0}&\bigg(\sum_{0\le k\le n}(-1)^km^{n-k}L_k\bigg)z^n=
\sum_{k,l\ge0}(-1)^kL_kz^km^{l}z^l\\
&=\frac1{1-mz}\frac{2+z}{1+z-z^2}\\&=\frac1{m^2+m-1}\frac{(m-2)-z(2m+1)}{1+z-z^2}-\frac {m(2m+1)}{m^2+m-1}\frac1{1-mz},
\end{align*}
which leads to
\begin{equation*}
\sum_{0\le k\le n}(-1)^km^{n-k}L_k=\frac{(-1)^{n+1}}{m^2+m-1}[(m-2)F_{n+1}-(2m+1)F_n]+\frac {m(2m+1)}{m^2+m-1}m^n.
\end{equation*}
Formul\ae{} (\ref{alterIdentGeneral}) and (\ref{alterIdent}) follow from these two in a straightforward way.

\end{document}